\documentclass[12pt,a4paper]{amsart}

\usepackage{euscript,amsfonts,amssymb,amsmath,amscd,amsthm,enumerate,hyperref,mathrsfs}

\usepackage{tikz,bm,float}
\usetikzlibrary{calc}
\colorlet{lgray}{white!85!black}
\colorlet{lred}{white!75!red}

\usepackage{comment}

\usepackage{graphicx}

\usepackage{color}

\usepackage[margin=1.1in]{geometry}
\newtheorem{theorem}{Theorem} 
\newtheorem*{theorem*}{Theorem}
\newtheorem{lemma}[theorem]{Lemma}

\newtheorem{proposition}[theorem]{Proposition}

\newtheorem{corollary}[theorem]{Corollary}

\theoremstyle{remark}
\newtheorem{remark}[theorem]{Remark}

\numberwithin{equation}{section} \numberwithin{theorem}{section}

\newcommand{\N}{\mathbb N}

\newcommand{\Z}{\mathbb Z}

\newcommand{\R}{\mathbb R}

\usepackage{MnSymbol}

\usepackage{skak}

\usepackage{adjustbox}
\usetikzlibrary{arrows}
\usetikzlibrary{decorations.pathreplacing}

\begin{document}

\title{Local central limit theorem for Mallows measure}

\author[Alexey Bufetov]{Alexey Bufetov}\address{Institute of Mathematics, Leipzig University, Augustusplatz 10, 04109 Leipzig, Germany.}\email{alexey.bufetov@gmail.com}

\author[Kailun Chen]{Kailun Chen}\address{Institute of Mathematics, Leipzig University, Augustusplatz 10, 04109 Leipzig, Germany.}\email{kailunxsx@gmail.com}

\begin{abstract}
We study the statistics of the Mallows measure on permutations in the limit pioneered by Starr (2009). Our main result is the local central limit theorem for its height function. We also re-derive versions of the law of large numbers and the large deviation principle, obtain the standard central limit theorem from the local one, and establish a multi-point version of the local central limit theorem.  
\end{abstract}

\maketitle

\section{Introduction}

\subsection{Main result}

Let $q \in [0;1)$. The Mallows measure on the set of permutations of size $N$ assigns a permutation $w$ the probability $q^{\mathrm{inv}(w)} Y_N$, where $\mathrm{inv}(w)$ is the number of inversions in $w$ and $Y_N$ is a normalization constant. For a permutation $w$ we define a \textit{height function}:
\begin{align*}
H_{L,K}(\omega) := \#\{ i : 1 \le i \le L, 1 \le w(i) \le K \}, \ \text{where}\ 1 \le L \le N,\ 1 \le K \le N. 
\end{align*}
Let us also define two functions $(0;1)^2 \to \R$:
\begin{align*}
h_{\beta}(x, y) := \frac{1}{\beta}\left[\ln(1-e^{-\beta})-\ln(e^{-\beta x}+e^{-\beta y}-e^{-\beta(x+y)}-e^{-\beta})\right],
\end{align*}
\begin{align*}
\sigma_{\beta}(x,y):=\sqrt{ \frac{ \beta (1- e^{-\beta})}{ (1- e^{- \beta(x-h_{\beta}(x, y))}) (1- e^{- \beta(y-h_{\beta}(x, y))}) } }.
\end{align*}

The main result of this paper is the following local central limit theorem.
\begin{theorem}
\label{th:main-intro}
Let $\alpha \in \R$, let $\beta \in \R_{>0}$, let $x,y \in (0;1)$, let $q=1-\beta/N$, and let $w$ be the random permutation of size $N$ distributed according to the Mallows measure. One has
$$
\lim_{N \to \infty} \frac{\sqrt{2 \pi N}}{ \sigma_{\beta}(x,y) }  \mathrm{Prob} \left( H_{\lfloor xN \rfloor,\lfloor yN \rfloor}(\omega) = \lfloor h_{\beta}(x, y) N + \alpha \sqrt{N} \rfloor \right) = \exp \left( -\frac{ \sigma_{\beta}^2(x,y) }{2} \alpha^2 \right). 
$$
\end{theorem}

On the way to the local central limit theorem (Theorem \ref{thm:LLT} below), we re-derive versions of the law of large numbers and the large deviation principle obtained by Starr \cite{Sta09} and Starr-Walters \cite{SW18} (Theorem \ref{thm:LDP}). We also establish the multi-point version of Theorem \ref{th:main-intro} (Theorem \ref{th:CLT-multi}), and the standard central limit theorem for $H_{\lfloor xN \rfloor,\lfloor yN \rfloor}$ (Corollary \ref{eq:CLT-global}). 

\subsection{Motivation and earlier work}

The Mallows measure was introduced in \cite{Mal57} among other interesting models of random permutations. In addition to its intrinsic elegance, it appears naturally in several contexts, such as the asymmetric simple exclusion process (=ASEP) and its mixing times (\cite{DR00}, \cite{BBHM05}, \cite{BN22}, \cite{BB24}), random walks on Hecke algebras (\cite{Buf20}, \cite{HS23}), binary search trees (\cite{ABC21}), stable matchings (\cite{AHHL21}), a toy model for random band matrices (\cite{GP18}), finitely dependent colorings of integers (\cite{HHL20}), q-exchangeability (\cite{GO10}, \cite{GO12}, \cite{BC24}), and statistical physics (\cite{Sta09}, \cite{SW18}). 

Various statistics of the Mallows measure and its asymptotics were studied. The behavior of the function $H_{L,K}(\omega)$ is related to a permuton-type limit, which was studied, in particular, in \cite{Sta09}, \cite{Muk16}, \cite{SW18}, \cite{KKRW20}, \cite{GK24}. In these works, the Law of Large Numbers was obtained for $H_{L,K}(\omega)$, as well as for a number of more general models.  The cycle structure of random Mallows permutations was studied in \cite{GP18}, \cite{HMV23}. The longest increasing subsequences were studied in \cite{MS13}, \cite{BP15}, \cite{BB17}. A stochastic process related with changing $q$ was studied in \cite{Cor22}, \cite{AK24}. 
 
\subsection{Approach}

We prove Theorem \ref{th:main-intro} and its generalizations via exact distribution formulas for $H_{L,K}(\omega)$, see Propositions \ref{prop:Height-Function}, \ref{prop:MalHeightInfManyPoints} below. These exact formulas are based on q-combinatorics behind the measure. The asymptotic analysis of the formulas involves various properties of the classical dilogarithm function, as well as somewhat delicate direct calculations.  

These exact distribution formulas for the Mallows measure do not seem to be extensively used for the asymptotic analysis in the existing literature. In particular, the works \cite{Sta09}, \cite{SW18}, as well as all previously cited works, relied on other methods. However, Proposition \ref{prop:Height-Function} was observed and proved in \cite{SW18}. It was also observed that ''The advantage of this second approach is that...one may prove local limit theorems'', see \cite[Appendix B]{SW18}. 
A part of computations of the present paper is probably alluded to in \cite{SW18}, \cite{Wal15}, where the q-Stirling formula is mentioned as a possible way towards the local limit theorems. However, we were unable to find further details of this approach in the existing literature. We also remark that the multidimensional exact formula (Proposition \ref{prop:MalHeightInfManyPoints}) seems to be a new result.

\subsection*{Acknowledgments}
Both authors were partially supported by the European Research Council (ERC), Grant Agreement No. 101041499.

\section{Preliminaries}
\label{sec:prelim}

\subsection{Mallows measure}
\label{ssec:MalDef}

We will use the notation $\Z_N := \{1,2, \dots, N \}$. Let $q$ be a real number, $0 \le q < 1$.

The Mallows measure $\mathcal{M}_N^{q}$ on the set of permutations $\Z_N \mapsto \Z_N$ assigns a permutation $w$ the probability $q^{\mathrm{inv}(w)} Y_N$, where $\mathrm{inv}(w)$ is the number of inversions in $w$ and $Y_N$ is a normalisation constant given by
$$
Y_N := \prod_{j=1}^N \frac{1-q}{1-q^j}.
$$
Another description of the measure $\mathcal{M}_N^{q}$ can be given with the use of the so called \textit{q-exchangeability}. Let $w(i) = a$, $w(i+1) = b$, and let $w_{i,i+1}$ coincides with $w$ for all integers with the exception that $w_{i,i+1} (i) = b$, $w_{i,i+1} (i+1) = a$. Then for any $i,a,b$ one has
$$
q \mathcal{M}_N^{q} (w) = \mathcal{M}_N^{q} (w_{i,i+1}) , \ \ \ \mbox{if $a<b$} ; \qquad \qquad  \mathcal{M}_N^{q} (w) = q \mathcal{M}_N^{q} (w_{i,i+1}) , \ \ \ \mbox{if $a>b$},
$$
and $\mathcal{M}_N^{q}$ is a unique probability measure on permutations that satisfies the property.

\subsection{The q-shuffle algorithm}
\label{ssec:q-shuffle}

The random permutation distributed according to the measure $\mathcal{M}_N^{q}$ can be efficiently sampled via the finite $q$-shuffle algorithm. Let $G_{N,q}$ be a truncated geometric distribution on $\Z_N$ given by
$$
G_{N,q} (i) = \frac{q^{i-1} (1-q)}{1-q^N}, \qquad i=1,2,\dots,N.
$$
Let $\xi_1, \xi_2, \dots, \xi_{N}$ be independent random variables, distributed according to $G_{N,q}$, $G_{N-1,q}$, ..., $G_{1,q}$, respectively.
The algorithm below also uses a (dynamically changing) word which consists of integers from $\Z_N$. Its initial form is $1 2 3 ... N$ (all integers are written in their linear order from left to right; each integer occurs exactly once), and this word will change during the algorithm.

At step 1 of the sampling algorithm we set $w(1)= \xi_1$. Then we remove the integer $w(1)$ from the word $1 2 3 4 \dots N$. At step 2, we set $w(2)$ to be equal to the integer which stands at position $\xi_2$ (counting from the left) at the current word (at step 2 it can be $\xi_2+1$ or $\xi_2$, depending on whether $\xi_1 \le \xi_2$ or not). Then we remove the chosen letter from the word, and iterate the procedure. At step number $k$, we set $w(k)$ to be equal to the integer which stands at position $\xi_k$ (counting from the left) at the current word, and then remove the chosen integer from the word. At the end, we obtain a random permutation $w(1) w(2) \dots w(N)$, and it is well-known (and easy to check) that it indeed follows the Mallows distribution.

\subsection{Dilogarithm function}
In this section, we recall the definition of the dilogarithm function and dilogarithm identities which will be useful for us. We refer to \cite{Kir95} for more details.

The Euler dilogarithm function $Li_{2}(z)$ is defined by a power series
\begin{align*}
Li_{2}(z) :=\sum_{n=1}^{\infty}\frac{z^{n}}{n^2}, \quad \text{for} \quad |z|<1.
\end{align*}
With the use of the integral representation
\begin{align}
\label{def:Li2}
Li_{2}(z)=-\int_{0}^{z}\frac{\ln(1-t)}{t}dt, \quad \text{for} \quad \mathbb{C} \setminus [1, +\infty),
\end{align}
function $Li_{2}(z)$ can be analytically continued as a  complex funtion to the complex plane cut along the real axis from 1 to $+\infty$. The dilogarithm function and the ordinary logarithm are related by
\begin{align}
\label{eq:Lilog}
\frac{d}{dz}Li_{2}(z)=-\frac{1}{z}\ln(1-z).
\end{align}
We will use the reflection property (L. Euler, 1768):
\begin{align}
\label{eq:inverse}
Li_2\left(\frac{1}{z}\right)=-Li_2(z)-\frac{\pi^2}{6}-\frac{1}{2}\ln^2(-z),
\end{align}
and the following nine-term functional equation (W. Mantel, 1898):
\begin{multline}
\label{eq:Mantel}
Li_2\left(\frac{ab}{uv}\right)=Li_2\left(\frac{a}{u}\right)+Li_2\left(\frac{b}{v}\right)+Li_2\left(\frac{a}{v}\right)+Li_2\left(\frac{b}{u}\right)+Li_2\left(u\right)+Li_2\left(v\right)\\
-Li_2\left(a\right)-Li_2\left(b\right)+\frac{1}{2}\ln^{2}\left(-\frac{u}{v}\right), \quad \text{for} \quad (1-a)(1-b)=(1-u)(1-v).
\end{multline}

\section{Exact formula for the height function of Mallows measure}

\subsection{Single-point height function}

Let $w$ be the random permutation distributed according to the Mallows measure $\mathcal{M}_{N}^{q}$, $N \in \mathbb{N}$. Let us define a \textit{height function} via
\begin{align}
\label{def:Height-Function}
H^{(N)}_{L,K}(\omega) := \#\{ i : 1 \le i \le L, 1 \le w(i) \le K \}, \ \text{where}\ 1 \le L \le N,\ 1 \le K \le N. 
\end{align}

The Mallows measure is invariant under the inversion map $\omega \to \omega^{-1}$. Therefore, the distribution of the height function \eqref{def:Height-Function} has the following symmetry property:
\begin{align}
\label{property:symmetry}
\mathcal{M}_N^{q}(H^{(N)}_{L,K}(\omega)=s)=\mathcal{M}_{N}^{q}(H^{(N)}_{L,K}(\omega^{-1})=s)=\mathcal{M}_N^{q}(H^{(N)}_{K,L}(\omega)=s),
\end{align}
where $s \in \mathbb{N}$ and $L+K-N \leq s \leq \min\{L,K\}$. 

Below we use a standard notation for $q$-Pochhammer symbols:
\begin{align}
\label{eq:qPoch}
(a ; q)_{\infty}=\prod_{k=0}^{\infty}\left(1-aq^{k}\right), \quad (a ; q)_{n}=\prod_{k=0}^{n-1}\left(1-aq^{k}\right)=\frac{(a ; q)_{\infty}}{(aq^{n} ; q)_{\infty}}.
\end{align}

The main goal of this section is to prove the following result. It was first proved (in a slightly different form) in \cite[Lemma 8.1]{SW18}.
\begin{proposition}
\label{prop:Height-Function}
Let $w$ be the random permutation distributed according to the Mallows measure $\mathcal{M}_{N}^{q}$, $N \in \mathbb{N}$. Fix $L,K$ such that $1 \le L \le N$, $1 \le K \le N$. For any $s \in \mathbb{N}$ such that $\max\{L+K-N, 0\} \leq s \leq \min\{L,K\}$, we have 
\begin{multline}
\label{eq:Height-Function}
\mathcal{M}_{N}^{q}\left( H^{(N)}_{L,K}(\omega)=s \right)\\
=q^{(K-s)(L-s)}\frac{(q^{s+1} ; q)_{\infty}(q^{N+1} ; q)_{\infty}(q^{K-s+1} ; q)_{\infty}(q^{L-s+1} ; q)_{\infty}(q^{N+s+1-K-L} ; q)_{\infty}}{(q^{K+1} ; q)_{\infty}(q^{L+1} ; q)_{\infty}(q^{N-K+1} ; q)_{\infty}(q^{N-L+1} ; q)_{\infty}(q ; q)_{\infty}}
\\=q^{(K-s)(L-s)}\frac{(q ; q)_{K}(q ; q)_{L}(q ; q)_{N-K}(q ; q)_{N-L}}{(q ; q)_{s}(q ; q)_{K-s}(q ; q)_{L-s}(q;q)_{N+s-K-L}(q ; q)_{N}}.\\
\end{multline}
\end{proposition}

\begin{proof}
For the proof we need to analyze the finite $q$-shuffle algorithm from Section \ref{ssec:q-shuffle}. We are interested in the prefix $\omega(1)\omega(2)\cdots\omega(L)$ and we need to distinguish whether each letter in this word is less or equal to $K$ or greater than $K$. We will denote by $\alpha$ the integers from 1 to $K$, and by $\beta$ the integers from $K+1$ to $N$. In these notations, the starting word involved in the finite $q$-shuffle algorithm is $\alpha\alpha\cdots\alpha\beta\beta\cdots\beta$ with $K$ letters $\alpha$ and $N-K$ letters $\beta$. We need to run $L$ steps of the sampling algorithm and keep track of how the words evolve.

In more detail, we regard $\omega(1)\omega(2)\cdots\omega(L)$ as a random word in $\{\alpha, \beta\}^{L}$. We obtain this word via the following procedure: the first letter $\omega(1)$ is generated to be $\alpha$ with probability $\frac{1-q^{K}}{1-q^{N}}$, and is generated to be $\beta$ with probability $\frac{q^{K}-q^{N}}{1-q^{N}}$; for $i=2,\cdots,L$, we generate $\omega(i)$ to be $\alpha$ with probability $\frac{1-q^{K_i}}{1-q^{N_i}}$, and generate $\omega(i)$ to be $\beta$ with probability $\frac{q^{K_i}-q^{N_i}}{1-q^{N_i}}$, where $N_i=N_{i-1}-1$ and 
\begin{align*}
K_i=
 \left\{ \begin{array}{rcl}
K_{i-1}-1, & \mbox{for} & \omega(i-1)=\alpha, \\
K_{i-1}, & \mbox{for} & \omega(i-1)=\beta.
\end{array}\right.
\end{align*}
In the above procedure, we use the notation $K_1=K, N_1=N$. It is readily verified that the above procedure corresponds to the $q$-shuffle algorithm. In particular, $K_i$ tracks the number of remaining integers which are $\leq K$ in the word at each step of the sampling, and $N_i$ tracks the total length of the word. This procedure assigns to each element of $\{\alpha, \beta\}^{L}$ a probability, and we need to sum these probabilities over all words with exactly $s$ letters $\alpha$ in order to get the claim of Proposition \ref{prop:Height-Function}. Thanks to the $q$-exchangeability property of the Mallows measure, we only need to calculate the generated probability of the word $\nu(\lambda):=\underbrace{\alpha\alpha\cdots\alpha}_{s}\underbrace{\beta\beta\cdots\beta}_{L-s}$, where we denote $\lambda:=(s, L-s)$. 

According to the above procedure, the generated probability of the word $\nu(\lambda)$ is equal to 
\begin{align}
\label{eq:nu-lambda}
\prod_{i=0}^{s-1}\frac{1-q^{K-i}}{1-q^{N-i}}\prod_{j=s}^{L-1}\frac{q^{K-s}-q^{N-j}}{1-q^{N-j}}=q^{(K-s)(L-s)} \frac{(q;q)_{K}(q;q)_{N-K}(q;q)_{N-L}}{(q;q)_{K-s}(q;q)_{N+s-K-L}(q;q)_{N}}.
\end{align}
Next, we need to calculate the summation $\sum_{\{\nu\}} q^{\text{inv}(\nu)}$, where $\{\nu\}$ stands for the set of all words $\nu$ belonging to the orbit of the word $\nu(\lambda)$ under the action of a permutation group. We use the MacMahon's formula for the generating function of the number of inversions in permutations of a multiset (see \cite{And98}): 
\begin{align}
\label{eq:inv}
\sum_{\{\nu\}} q^{\text{inv}(\nu)}=\frac{(q;q)_{L}}{(q;q)_{s}(q;q)_{L-s}}.
\end{align}
Combining \eqref{eq:nu-lambda} and \eqref{eq:inv}, we get the last formula of \eqref{eq:Height-Function}. The middle one is obtained from it via a simple transformation. 
\end{proof}

%

\subsection{Multi-point height function}

Let $w$ be the random permutation distributed according to the Mallows measure $\mathcal{M}_{N}^{q}$, $N \in \mathbb{N}$. Let us define a more general height function via
$$
H^{(N)}_{[L,\hat L]; K}(\omega) := \#\{ i : L \le i \le \hat L, 1 \le w(i) \le K \}, \ \  \text{where}\ 1 \le L \le \hat L \le N,\ 1 \le K \le N, 
$$
and let us define a \textit{multi-point height function} as a vector of random variables of the following form
\begin{equation}
\label{deff:MalHeightInfManyPoints}
H^{(N)}_{L_1, L_2, \dots, L_r; K} (w ) \\ := \left( H^{(N)}_{L_1,K}(\omega), H^{(N)}_{[L_1+1,L_2],K}(\omega), \dots, H^{(N)}_{[L_{r-1}+1,L_r],K}(\omega) \right), \ \ r \in \N.
\end{equation}
The main goal of this section is to prove the following result.

\begin{proposition}
\label{prop:MalHeightInfManyPoints}
Let $w$ be the random permutation distributed according to the Mallows measure $\mathcal{M}_{N}^{q}$, $N \in \mathbb{N}$.
For any $r \in \N$, $1 \le K \le N$, $1 \le L_1 \le \dots \le L_i \le \dots L_r \le N$, and any $(s_1, s_2, \dots, s_r) \in \{0,1,2, \dots, \}^r$, such that $\max\{L_i-L_{i-1}+K-N, 0\} \leq s_i \leq \min\{L_i-L_{i-1},K\}$, we have
\begin{multline}
\label{eq:MalHeightInfManyPoints}
\mathcal{M}_{N}^{q} \left( H^{(N)}_{L_1, L_2, \dots, L_r; K} (w ) = (s_1,s_2, \dots, s_r) \right) \\
= \mathcal{M}_{N}^{q} \left( H^{(N)}_{L_1,K } = s_1  \right) \mathcal{M}_{N}^{q} \left( H^{(N-L_1)}_{L_{2} -L_1, K-s_1} = s_2 \right) \cdots \mathcal{M}_{N}^{q} \left( H^{(N-L_{r-1})}_{L_{r} - L_{r-1} , K-s_1-s_2-\dots-s_{r-1}} = s_r  \right).
\end{multline}
\end{proposition}

\begin{remark}
The right-hand side of \eqref{eq:MalHeightInfManyPoints} is an explicit product formula due to \eqref{eq:Height-Function}.
\end{remark}

\begin{proof}

For the proof we need to use the sampling algorithm from Section \ref{ssec:q-shuffle} again. We start with the prefix of the permutation $w(1) w(2) \dots w(L_1)$. The chance that exactly $s_1$ of these letters are $\le K$ equals $ \mathcal{M}_{N}^{q} \left( H^{(N)}_{L_1,K } = s_1  \right)$ immediately from the definition of the algorithm. Therefore, $(K-s_1)$ slots remain free when we start sampling the places of integers from $[L_1+1;L_2]$. Since the random variables which are used in this sampling are independent from the ones used in the first step, we obtain that the chance that exactly $s_2$ integers from  
$[L_1+1;L_2]$ occupy these spots is equal to $\mathcal{M}_{N}^{q} \left( H^{(N-L_1)}_{L_{2} -L_1, K-s_1} = s_2 \right)$. Iterating this argument, we arrive at the statement of the proposition.

\end{proof}

In the remainder of the text, we will often omit the upper index $(N)$ in the notation $H^{(N)}_{K,L}$ for brevity.

\section{Asymptotics of height functions}

\subsection{Asymptotics of $q$-Pochhammer functions}
In this section, we consider the asymptotics of $q$-Pochhammer functions under the following scaling:
\begin{align}
\label{eq:scale}
\beta \in \R_{>0}, \quad q=1-\frac{\beta}{N}, \quad N \to \infty.
\end{align}

The following asymptotics is well-known (see \cite{HR18}). 
\begin{proposition}
Under the scaling \eqref{eq:scale}, we have
\begin{align}
\label{eq:lnqq}
\ln (q ; q)_{\infty}=-\frac{\pi^2}{6\beta}N+\frac{\pi^2}{12}-\frac12\ln\frac{\beta}{2\pi N}+O(N^{-1}).
\end{align}
\end{proposition}


The main result of this section is the following proposition. Similar asymptotics of $q$-Pochhammer functions were performed in various works, including, e.g., \cite{Moa84}, \cite{BC14}. We were unable to locate precisely this statement.
\begin{proposition}
Let $\alpha = \alpha(N)$ be a sequence such that $\alpha = o \left( N^{1/6} \right)$, let $\delta \in \R_{>0}$, and let $s=\lfloor \delta N+\alpha \sqrt{N} \rfloor$. 
Under the scaling \eqref{eq:scale}, we have the following asymptotics:
\begin{multline}
\label{eq:lnqsqrtNq}
 \ln (q^{s+1} ; q)_{\infty}=-\frac{Li_{2}(e^{-\beta\delta})}{\beta} \cdot N -\alpha\ln(1- e^{-\delta \beta})\cdot \sqrt{N}-\alpha^2 \cdot \frac{\beta e^{-\delta\beta}}{2(1-e^{-\delta\beta})}\\
+\frac12Li_{2}(e^{-\beta\delta})-\frac{1+\delta \beta}{2}\ln(1- e^{-\delta \beta})+O\left(\frac{\alpha^3}{\sqrt{N}}\right)+O\left(\frac{\alpha^4}{N}\right).
\end{multline}
\end{proposition}

\begin{proof}
Let $f(x)$ be a continuously differentiable function $[a,b] \to \R$, where $a,b \in \N$. A simple version of the Euler-Maclaurin formula reads
$$
\sum_{n=a}^{b}f(n)=\int_{a}^{b}f(x)dx+\frac{f(a)+f(b)}{2}+ \int_a^b f'(x) \left( x - \lfloor x \rfloor - \frac12 \right)dx.
$$
Applying it to the logarithm of the $q$-Pochhammer symbol, we have:
\begin{multline*}
 \ln (q^{s+1} ; q)_{\infty} = \sum_{n=0}^{\infty} \ln (1- e^{(n+\delta N+\alpha\sqrt{N}+1)\ln(1-\frac{\beta}{N})} )\\
=\int_{0}^{\infty}\ln(1- e^{(t+\delta N+\alpha\sqrt{N}+1)\ln(1-\frac{\beta}{N})})dt + \frac{1}{2}\ln(1- e^{(\delta N+\alpha\sqrt{N}+1)\ln(1-\frac{\beta}{N})})+O(N^{-1})\\
=\frac{1}{-\ln(1-\frac{\beta}{N})}\int_{0}^{\infty}\ln(1- e^{(\delta N+\alpha\sqrt{N}+1)\ln(1-\frac{\beta}{N})-r})dr + \frac{1}{2}\ln(1- e^{-\delta\beta})+O(N^{-1})\\
=\frac{1}{-\ln(1-\frac{\beta}{N})}\int_{-(\delta N+\alpha\sqrt{N}+1)\ln(1-\frac{\beta}{N})}^{\infty}\ln(1- e^{-z})dz + \frac{1}{2}\ln(1- e^{-\delta\beta})+O(N^{-1})\\
=\frac{\int_{\delta \beta}^{\infty}\ln(1- e^{-z})dz-\int_{\delta \beta}^{-(\delta N+\alpha\sqrt{N}+1)\ln(1-\frac{\beta}{N})}\ln(1- e^{-z})dz}{-\ln(1-\frac{\beta}{N})} + \frac{1}{2}\ln(1- e^{-\delta\beta})+O(N^{-1})
\end{multline*}
Using the Taylor expansion of $g(z)=\ln(1- e^{-z})$ at $z=\delta\beta$, we have the following asymptotics:
\begin{multline*}
\int_{\delta \beta}^{-(\delta N+\alpha\sqrt{N}+1)\ln(1-\frac{\beta}{N})} \ln(1- e^{-z}) dz=\sum_{k=0}^{\infty}\frac{g^{(k)}(\delta\beta)}{k!}\int_{\delta \beta}^{-(\delta N+\alpha\sqrt{N}+1)\ln(1-\frac{\beta}{N})}(z-\delta\beta)^k dz\\
=\left[\alpha \cdot \frac{\beta}{\sqrt{N}}+\left(1+\frac{\delta\beta}{2}\right) \cdot \frac{\beta}{N} +\frac{\alpha\beta}{2\sqrt{N}}\frac{\beta}{N}\right] \cdot \ln(1- e^{-\delta\beta})+\alpha^2 \cdot \frac{\beta e^{-\delta\beta}}{2(1-e^{-\delta\beta})} \cdot \frac{\beta}{N}\\
+\left(1+\frac{\delta\beta}{2}\right) \frac{\beta e^{-\delta\beta}}{1-e^{-\delta\beta}} \frac{\alpha\beta}{N\sqrt{N}}-\frac{e^{-z}}{6(1- e^{-z})^2}\frac{\alpha^3\beta^3}{N\sqrt{N}}+O\left(\frac{\alpha^4}{N^{2}}\right),
\end{multline*}
where we used the expansion $-\ln(1-\frac{\beta}{N})=\sum_{k=1}^{\infty}\frac{\beta^k}{kN^k}$. Then we have 
\begin{multline*}
 \ln (q^{s+1} ; q)_{\infty}
 =\left( \beta^{-1}N-\frac12 \right) \int_{\delta \beta}^{\infty} \ln(1- e^{-z})dz+ \frac{1}{2}\ln(1- e^{-\delta\beta})+O(N^{-1})- \left( \beta^{-1}N-\frac12 \right) O\left(\frac{\alpha^4}{N^{2}}\right)\\
-\left( \beta^{-1}N-\frac12 \right) \left[\alpha \cdot \frac{\beta}{\sqrt{N}}+\left(1+\frac{\delta\beta}{2}\right) \cdot \frac{\beta}{N} +\frac{\alpha\beta}{2\sqrt{N}}\frac{\beta}{N}\right] \cdot \ln(1- e^{-\delta\beta})+(\beta^{-1}N-\frac12)\frac{e^{-z}}{6(1- e^{-z})^2}\frac{\alpha^3\beta^3}{N\sqrt{N}}\\
-\alpha^2 \left( \beta^{-1}N-\frac12 \right) \cdot \frac{\beta e^{-\delta\beta}}{2(1-e^{-\delta\beta})} \cdot \frac{\beta}{N} - \left(\beta^{-1}N-\frac12 \right) \left(1+\frac{\delta\beta}{2}\right) \frac{\beta e^{-\delta\beta}}{1-e^{-\delta\beta}} \frac{\alpha\beta}{N\sqrt{N}}\\
=\beta^{-1}N\int_{\delta \beta}^{\infty}\ln(1- e^{-z})dz-\alpha\ln(1- e^{-\delta \beta})\sqrt{N}-\alpha^2 \cdot \frac{\beta e^{-\delta\beta}}{2(1-e^{-\delta\beta})}\\
-\frac12\int_{\delta \beta}^{\infty}\ln(1- e^{-z})dz-\frac{1+\delta \beta}{2}\ln(1- e^{-\delta \beta})+O\left(\frac{\alpha^3}{\sqrt{N}}\right)+O\left(\frac{\alpha}{\sqrt{N}}\right)+O\left(\frac{\alpha^4}{N}\right),
\end{multline*}
where we use the fact $-1/\ln(1-\frac{\beta}{N})=\beta^{-1}N-\frac12+O(N^{-1})$. Using also \eqref{def:Li2}, we arrive at \eqref{eq:lnqNq}.
\end{proof}

\begin{corollary}
Under the scaling \eqref{eq:scale}, we have:
\begin{align}
\label{eq:lnqNq}
 \ln (q^{\delta N+1} ; q)_{\infty}=-\frac{Li_{2}(e^{-\beta\delta})}{\beta} \cdot N+\frac12Li_{2}(e^{-\beta\delta})-\frac{1+\beta \delta }{2}\ln(1- e^{-\beta \delta })+O(N^{-1}).
\end{align}
\end{corollary}

\subsection{Large deviation principle}

In this section we briefly sketch the proofs of the large deviation principle and the law of large numbers for the height function \eqref{def:Height-Function}, first obtained in \cite{SW18}, \cite{Muk16}, \cite{Sta09}. 

\begin{theorem}
\label{thm:LDP}
Fix $\beta \in \mathbb{R}_{>0}$, and let $\omega$ be the random permutation distributed according to the Mallows measure $\mathcal{M}_{N}^{q}$, $q=1-\beta/N$, $N \in \mathbb{N}$. Fix $x,y\in(0,1)$, and define the function $h_{\beta}(x, y)$ via:
\begin{align}
\label{eq:mean}
h_{\beta}(x, y)=\frac{1}{\beta}\left[\ln(1-e^{-\beta})-\ln(e^{-\beta x}+e^{-\beta y}-e^{-\beta(x+y)}-e^{-\beta})\right].
\end{align}
Then for any $\delta \in (h_{\beta}(x, y), \min\{x,y\}]$, the following convergence holds:
\begin{align}
\label{eq:LDP1}
\lim_{N\to\infty}\frac{1}{N} \ln \mathcal{M}_{N}^{q} \left( H_{x N, y N}(\omega)\geq\delta N \right)=-a_{\beta}(x,y; \delta),
\end{align}
and for any $\delta \in [\max\{x+y-1,0\}, h_{\beta}(x, y))$, we have 
\begin{align}
\label{eq:LDP2}
\lim_{N\to\infty}\frac{1}{N} \ln \mathcal{M}_{N}^{q} \left( H_{x N, y N}(\omega) \leq \delta N \right)=-a_{\beta}(x,y; \delta),
\end{align}
where the rate function $a_{\beta}(x,y; \delta)$ is given by:
\begin{multline}
\label{eq:rate}
a_{\beta}(x, y;\delta)=\frac{1}{\beta}\bigg(\beta^2(x-\delta)(y-\delta)-\frac{\pi^2}{6}+Li_{2}(e^{-\beta\delta })+Li_{2}(e^{- \beta(x-\delta)})+Li_{2}(e^{- \beta(y-\delta)})\bigg.\\
\bigg.+Li_{2}(e^{- \beta(1-x-y+\delta)})+Li_{2}(e^{-\beta})-Li_{2}(e^{-\beta x })-Li_{2}(e^{- \beta y })-Li_{2}(e^{-\beta(1-x) })-Li_{2}(e^{-\beta(1-y)})\bigg).
\end{multline}
\end{theorem}

Before giving a proof to Theorem \ref{thm:LDP}, we establish some properties of functions $h_{\beta}(x, y)$ and $a_{\beta}(x,y; \delta)$.
\begin{lemma}
\label{lem:Li2}
For $\delta=h_{\beta}(x, y)$, one has:
\begin{multline}
\label{eq:Li2}
Li_{2}(e^{- \beta(1-x-y+\delta)})=\frac{\pi^2}{6}-\beta^2(x-\delta)(y-\delta)-Li_{2}(e^{-\beta\delta })-Li_{2}(e^{- \beta(x-\delta)})\\
-Li_{2}(e^{- \beta(y-\delta)})-Li_{2}(e^{-\beta})+Li_{2}(e^{-\beta x })+Li_{2}(e^{- \beta y })+Li_{2}(e^{-\beta(1-x) })+Li_{2}(e^{-\beta(1-y)}).
\end{multline}
\end{lemma}

\begin{proof}
Identity \eqref{eq:Li2} follows from the nine-term functional identity \eqref{eq:Mantel} by setting $a=e^{-\beta}$, $b=e^{-\beta h_{\beta}(x, y)}$, $u=e^{-\beta x}$, and $v=e^{-\beta y}$, where we also use the reflection property \eqref{eq:inverse}.
\end{proof}

\begin{lemma}
\label{lem:rate function}
For any fixed $\beta \in \mathbb{R}_{>0}$ and any $x,y\in(0,1)$, $a_{\beta}(x,y; \delta)$ is a convex function on $[\max\{x+y-1,0\}, \min\{x,y\}]$ with the following properties:
\begin{align}
\label{eq:zero}
&a_{\beta}(x,y; h_{\beta}(x, y))=0,\\
\label{eq:zero-first-order}
&\left.\frac{d}{d\delta}a_{\beta}(x,y; \delta)\right|_{h_{\beta}(x, y)}=0,\\
\label{eq:zero-second-order}
&\left.\frac{d^2}{d^2\delta}a_{\beta}(x,y; \delta)\right|_{h_{\beta}(x, y)}=\beta\exp(2d_{\beta}(x,y)),
\end{align}
where the function $d_{\beta}(x,y)$ is defined via
\begin{align}
\label{eq:variance}
d_{\beta}(x,y)=\frac{1}{2}\ln(1- e^{-\beta})-\frac{1}{2}\ln(1- e^{- \beta(x-h_{\beta}(x, y))})-\frac{1}{2}\ln(1- e^{-\beta(y-h_{\beta}(x, y)) }).
\end{align}
In particular, $a_{\beta}(x,y; \delta)\geq 0$ and takes its minimum value at a unique point $\delta=h_{\beta}(x, y)$.
\end{lemma}

\begin{proof}
Identity \eqref{eq:Li2} implies that $h_{\beta}(x, y)$ is the zero of $a_{\beta}(x,y; \delta)$; thus, \eqref{eq:zero} holds. With the use of \eqref{eq:Lilog} we obtain 
\begin{multline}
\label{eq:first-order-derivative}
\frac{d}{d\delta}a_{\beta}(x,y; \delta)=\beta(2\delta-x-y)-\ln(1- e^{-(x-\delta) \beta})-\ln(1- e^{-(y-\delta) \beta})\\
+\ln(1- e^{-\delta \beta})+\ln(1- e^{-(1-x-y+\delta) \beta}).
\end{multline}
It is a direct check that $h_{\beta}(x, y)$ is a zero of $\frac{d}{d\delta}a_{\beta}(x,y; \delta)$; thus, \eqref{eq:zero-first-order} holds. Further, the second derivative of $a_{\beta}(x,y; \delta)$ is given by the following formula:
\begin{align}
\label{eq:second-order-derivative}
\frac{d^2}{d^2\delta}a_{\beta}(x,y; \delta)=\beta\left(2+\frac{e^{-\delta\beta}}{1-e^{-\delta\beta}}+\frac{e^{-(x-\delta) \beta}}{1-e^{-(x-\delta) \beta}}+\frac{e^{-(y-\delta) \beta}}{1-e^{-(y-\delta) \beta}}+\frac{e^{-(1-x-y+\delta) \beta}}{1-e^{-(1-x-y+\delta) \beta}}\right).
\end{align}
The right-hand side of \eqref{eq:second-order-derivative} is positive for any $\beta \in \mathbb{R}_{>0}$ and $x,y\in(0,1)$; therefore, $a_{\beta}(x,y; \delta)$ is a convex function on $[\max\{x+y-1,0\}, \min\{x,y\}]$. When $\delta=h_{\beta}(x,y)$, it is a direct check that the following identity holds:
\begin{align}
\label{eq:d-beta-xy}
2+\frac{e^{-\delta\beta}}{1-e^{-\delta\beta}}+\frac{e^{-(x-\delta) \beta}}{1-e^{-(x-\delta) \beta}}+\frac{e^{-(y-\delta) \beta}}{1-e^{-(y-\delta) \beta}}+\frac{e^{-(1-x-y+\delta) \beta}}{1-e^{-(1-x-y+\delta) \beta}}=\exp(2d_{\beta}(x,y)),
\end{align}
which concludes the proof.
\end{proof}

\begin{proof}[Proof of Theorem \ref{thm:LDP}]
By the exact formula \eqref{eq:Height-Function}, we have\footnote{Here and in the sequel we omit symbols $\lfloor$, $\rfloor$ in the formulas for their better readability. For example, we write in the next formula $\delta N$ instead of $\lfloor \delta N \rfloor$}:
\begin{multline*}
\mathcal{M}_{N}^{q} \left( H_{x N, y N}(\omega)=\delta N \right)=\exp \bigg( (x-\delta)(y-\delta)N^2 \cdot \ln q  + \ln (q^{\delta N+1} ; q)_{\infty} + \ln (q^{(x-\delta)N+1} ; q)_{\infty} \bigg.\\
\bigg.   + \ln (q^{(y-\delta)N+1} ; q)_{\infty} + \ln (q^{(1-x-y+\delta)N+1} ; q)_{\infty} + \ln (q^{N+1} ; q)_{\infty}-\ln (q^{xN+1} ; q)_{\infty}\bigg.\\
\bigg.- \ln (q^{yN+1} ; q)_{\infty} - \ln (q^{(1-x)N+1} ; q)_{\infty} - \ln (q^{(1-y)N+1} ; q)_{\infty} - \ln (q ; q)_{\infty}\bigg).
\end{multline*}
Under the scaling \eqref{eq:scale}, we can get the following result by using the asymptotics \eqref{eq:lnqq} and \eqref{eq:lnqNq}:
\begin{align*}
\mathcal{M}_{N}^{q} \left( H_{x N, y N}(\omega)=\delta N \right)=\exp\left(-a_{\beta}(x,y; \delta)N+O(\ln(N))\right),
\end{align*}
where $a_{\beta}(x,y; \delta)$ is given by \eqref{eq:rate}. By Lemma \ref{lem:rate function}, we have:
\begin{multline*}
\lim_{N\to\infty}\frac{1}{N} \ln \mathcal{M}_{N}^{q} \left( H_{x N, y N}(\omega)=\delta N \right) \leq \lim_{N\to\infty}\frac{1}{N} \ln \mathcal{M}_{N}^{q} \left( H_{x N, y N}(\omega)\geq\delta N \right) \\
\leq \lim_{N\to\infty}\frac{1}{N} \ln \left( (\min\{x,y\}N+1) \mathcal{M}_{N}^{q} \left( H_{x N, y N}(\omega)=\delta N \right)\right)
\end{multline*}
for $\delta \in (h_{\beta}(x, y), \min\{x,y\}]$; this implies \eqref{eq:LDP1}. Analogously, we have:
\begin{multline*}
\lim_{N\to\infty}\frac{1}{N} \ln \mathcal{M}_{N}^{q} \left( H_{x N, y N}(\omega)=\delta N \right) \leq \lim_{N\to\infty}\frac{1}{N} \ln \mathcal{M}_{N}^{q} \left( H_{x N, y N}(\omega)\leq \delta N \right) \\
\leq \lim_{N\to\infty}\frac{1}{N} \ln \left( (\min\{x,y\}N+1) \mathcal{M}_{N}^{q} \left( H_{x N, y N}(\omega)=\delta N \right)\right),
\end{multline*}
for $\delta \in [\max\{x+y-1,0\}, h_{\beta}(x, y))$; that implies \eqref{eq:LDP2}.
\end{proof}

A corollary of Theorem \ref{thm:LDP} is the strong law of large numbers for the height function \eqref{def:Height-Function}.
\begin{corollary}
\label{cor:LLN}
Fix a parameter $\beta \in \mathbb{R}_{>0}$, and let $\omega$ be the random permutation distributed according to the Mallows measure $\mathcal{M}_{N}^{q}$, $N \in \mathbb{N}$. For any $0 < x, y <1$, the following convergence holds:
\begin{align}
\label{eq:LLN}
\frac{H_{x N, y N}(\omega)}{N}\xrightarrow{\text{a.s.}}h_{\beta}(x,y), \qquad \text{as}\ N\to\infty,
\end{align}
where the function $h_{\beta}(x, y)$ is given by \eqref{eq:mean}.
\end{corollary}

\begin{proof}
One can get the following result in a standard way by using \eqref{eq:LDP1} and \eqref{eq:LDP2}:
\begin{align*}
\sum_{N=1}^{\infty} \mathcal{M}_{N}^{q}  \left\{\bigg\vert \frac{H_{x N, y N}(\omega)}{N}-h_{\beta}(x, y) \bigg\vert \geq \epsilon \right\}<\infty.
\end{align*}
Then one obtains the strong law of large numbers \eqref{eq:LLN} by using the Borel-Cantelly lemma.
\end{proof}

\begin{remark}
The weak law of large numbers
\begin{align*}
\lim_{\epsilon \to 0} \lim_{N \to \infty} \mathcal{M}_{N}^{q}  \left\{\bigg\vert \frac{H_{x N, y N}(\omega)}{N}-h_{\beta}(x, y) \bigg\vert>\epsilon \right\}=0,
\end{align*}
was first established in \cite{Sta09}. One can recover the density function $u(x,y)$ from \cite[Theorem 1.1]{Sta09} with the use of $h_{\beta}(x, y)$ via
\begin{align*}
\frac{\partial^2 h_{\beta}(x, y)}{\partial x \partial y} =\frac{\frac{\beta}{2}\sinh\frac{\beta}{2}}{\left(e^{\frac{\beta}{4}}\cosh\frac{\beta(x-y)}{2}-e^{-\frac{\beta}{4}}\cosh\frac{\beta(x+y-1)}{2}\right)^2}=u(x,y).
\end{align*}
\end{remark}

\begin{remark}
The large deviation principle for the Mallows measure has been established in \cite{Muk16, SW18}.  We can match our rate function \eqref{eq:rate} with the function $\tilde{\Phi}(\theta_1,\theta_2;t_{11}, t_{12}, t_{21},t_{22})$ in \cite[Theorem 3.7]{SW18} in the following way:
\begin{align*}
a_{\beta}(x, y;\delta)=\tilde{\Phi}(x,y;\delta, x-\delta, y-\delta,1-x-y-\delta).
\end{align*}
\end{remark}

\begin{remark}
As noted in \cite{SW18}, one can also derive asymptotics of statistics of the Mallows measure by using the q-Stirling’s formula \cite{Moa84}. In our setting, the height function \eqref{def:Height-Function} can be rewritten in the following form:
\begin{align}
\label{eq:Height-Function-factorial}
\mathcal{M}_{N}^{q} (f_{N,K,L}(\omega)=s)=q^{(K-s)(L-s)}\frac{[K]!_{q}[L]!_{q}[N-K]!_{q}[N-L]!_{q}}{[s]!_{q}[K-s]!_{q}[L-s]!_{q}[N+s-K-L]!_{q}[N]!_{q}},
\end{align}
where we use the q-factorial $[n]!_{q}=\frac{(q ; q)_{n}}{(1-q)^n}$. Then one can perform the asymptotics by using the following q-Stirling’s formula:
\begin{align*}
\ln\left(\frac{[N]!_{q}}{N!}\right)=N\int_{0}^{1}\ln\left(\frac{1-e^{-\beta x}}{\beta x}\right)+\frac{\beta}{2}+\frac{1}{2}\ln\left(\frac{1-e^{-\beta}}{\beta}\right)+R_{n}(\beta),
\end{align*}
where we use the scaling \eqref{eq:scale}, and $R_{n}(\beta) \to 0$ as $N\to\infty$. We refer to \cite[Theorem 5.1.1]{Wal15} for a new proof of the q-Stirling’s formula. In this way, \cite[Lemma 5.6.2]{Wal15} provided a proof of the large deviation principle for the Mallows measure.
\end{remark}

\subsection{Local central limit theorem: One point}
Our main result is the following local central limit theorem for the height function \eqref{def:Height-Function}. We denote the density of the standard normal distribution by $\mathcal{N}$.
\begin{theorem}
\label{thm:LLT}
Fix a parameter $\beta \in \mathbb{R}_{>0}$, and let $\omega$ be the random permutation distributed according to the Mallows measure $\mathcal{M}_{N}^{q} $, where $q=1-\beta/N$ and $N \in \mathbb{N}$. Let $A_N$ be a non-decreasing sequence such that $\lim_{N\to\infty}A_N/N^{1/6}=0$. Then, for any $0 < x, y <1$, the following convergence holds:
\begin{align}
\label{eq:LLT}
\max_{|k-h_{\beta}(x, y) N| < \sqrt{N} \cdot A_N}\left| \frac{\sigma_{N, \beta}(x,y) \mathcal{M}_{N}^{q}\left( H_{x N, y N}(\omega)= k \right)}{\mathcal{N}\left(\frac{k-h_{\beta}(x, y) N}{\sigma_{N, \beta}(x,y)}\right)} - 1\right|=O\left(\frac{A_N^3}{\sqrt{N}}\right),
\end{align}
where the function $h_{\beta}(x, y)$ is given by \eqref{eq:mean}, $\sigma_{N, \beta}(x,y)=\frac{\sqrt{N}}{\sqrt{\beta}\exp(d_{\beta}(x,y))}$, and the function $d_{\beta}(x,y)$ is given by \eqref{eq:variance}.
\end{theorem}

\begin{remark}
Theorem \ref{th:main-intro} is an immediate corollary of Theorem \ref{thm:LLT}.
\end{remark}

Before proving Theorem \ref{thm:LLT}, we record the following property of function $h_{\beta}(x, y)$:
\begin{lemma}
\label{lem:lndelta}
For $\delta=h_{\beta}(x, y)$, one has:
\begin{multline}
\label{eq:c-beta-xy}
\frac{\beta^2}{2}(x+y)\delta-\beta^2xy-\frac{1}{2}\ln(1- e^{-\delta \beta})-\frac{x \beta}{2}\ln(1- e^{-(x-\delta) \beta})-\frac{y \beta}{2}\ln(1- e^{-(y-\delta) \beta})\\
-\frac{1+(1-x-y) \beta}{2}\ln(1- e^{-(1-x-y+\delta) \beta})+\frac{1+x \beta}{2}\ln(1- e^{-x \beta})+\frac{1+y\beta}{2}\ln(1- e^{-y \beta})\\
+\frac{1+(1-x)\beta}{2}\ln(1- e^{-(1-x)\beta})+\frac{1+(1-y)\beta}{2}\ln(1- e^{-(1-y)\beta})-\left(1+\frac{\beta}{2}\right)\ln(1- e^{-\beta})=0.
\end{multline}
\end{lemma}

\begin{proof}
For $\delta=h_{\beta}(x, y)$, one can check that the following identities hold:
\begin{align*}
& \ln(1- e^{-\beta\delta })=\ln(1- e^{-x \beta})+\ln(1- e^{-y \beta})-\ln(1- e^{-\beta}),\\
& \ln(1- e^{-(x-\delta) \beta})=\ln(1- e^{-x \beta})+\ln(1- e^{-(1-y) \beta})-\beta y-\ln(e^{-\beta x}+e^{-\beta y}-e^{-\beta(x+y)}-e^{-\beta}),\\
& \ln(1- e^{-(y-\delta) \beta})=\ln(1- e^{-(1-x) \beta})+\ln(1- e^{-y \beta})-\beta x-\ln(e^{-\beta x}+e^{-\beta y}-e^{-\beta(x+y)}-e^{-\beta}),\\
& \ln(1- e^{-(1-x-y+\delta) \beta})=\ln(1- e^{-(1-x) \beta})+\ln(1- e^{-(1-y) \beta})-\ln(1- e^{-\beta}).
\end{align*}
Then \eqref{eq:c-beta-xy} follows from the above identities by a direct calculation.
\end{proof}

\begin{proof}[Proof of Theorem \ref{thm:LLT}]
By the exact formula \eqref{eq:Height-Function}, we have:
\begin{multline*}
\mathcal{M}_{N}^{q} \left( H_{x N, y N}(\omega)=h_{\beta}(x, y) \cdot N  +  \alpha \sigma_{N, \beta}(x,y) \right)\\
=\exp \bigg( \left[(x-h_{\beta}(x, y))N- \alpha \sigma_{N, \beta}(x,y)\right] \left[(x-h_{\beta}(x, y))N- \alpha \sigma_{N, \beta}(x,y)\right] \cdot \ln q  \bigg.\\
\bigg. + \ln (q^{h_{\beta}(x, y)N+ \alpha \sigma_{N, \beta}(x,y)+1} ; q)_{\infty} + \ln (q^{(x-h_{\beta}(x, y))N- \alpha \sigma_{N, \beta}(x,y)+1} ; q)_{\infty} + \ln (q^{N+1} ; q)_{\infty} \bigg.\\
\bigg.+ \ln (q^{(y-h_{\beta}(x, y))N- \alpha \sigma_{N, \beta}(x,y)+1} ; q)_{\infty}+ \ln (q^{(1-x-y+h_{\beta}(x, y))N+ \alpha \sigma_{N, \beta}(x,y)+1} ; q)_{\infty} \bigg.\\
\bigg. -\ln (q^{xN+1} ; q)_{\infty}- \ln (q^{yN+1} ; q)_{\infty} - \ln (q^{(1-x)N+1} ; q)_{\infty} - \ln (q^{(1-y)N+1} ; q)_{\infty} - \ln (q ; q)_{\infty}\bigg)
\end{multline*}
Under the scaling \eqref{eq:scale}, we get\footnote{More details on this step can be found in the proof of Lemma \ref{lem:K-gamma} below.} the following result with the use of asymptotics \eqref{eq:lnqq}, \eqref{eq:lnqsqrtNq} and \eqref{eq:lnqNq}:
\begin{multline*}
\mathcal{M}_{N}^{q} \left( H_{x N, y N}(\omega)=h_{\beta}(x, y) \cdot N +  \alpha \sigma_{N, \beta}(x,y) \right)\\
=\frac{1}{\sigma_{N, \beta}(x,y) \sqrt{2\pi }}\exp\bigg(a_{\beta}(x,y)N + b_{\beta}(x,y)\sigma_{N, \beta}(x,y) \cdot \alpha - \frac{f_{\beta}(x,y)}{2\exp(2d_{\beta}(x,y))} \cdot \alpha^{2}  \bigg.\\
\bigg. -\frac{\beta}{2}a_{\beta}(x,y)+\frac{\beta h_{\beta}(x, y)}{2}b_{\beta}(x,y)+c_{\beta}(x,y)+O\left(\frac{\alpha^3}{\sqrt{N}}\right)+O\left(\frac{\alpha}{\sqrt{N}}\right)+O\left(\frac{\alpha^4}{N}\right)\bigg),
\end{multline*}
where the function $a_{\beta}(x,y)$ has the following expression:
\begin{multline*}
a_{\beta}(x,y)=\beta^{-1}\bigg(\frac{\pi^2}{6}-\beta^2(x-h_{\beta}(x, y))(y-h_{\beta}(x, y))-Li_{2}(e^{-\beta h_{\beta}(x, y)})\bigg.\\
\bigg.-Li_{2}(e^{- \beta(x-h_{\beta}(x, y))})-Li_{2}(e^{- \beta(y-h_{\beta}(x, y))})-Li_{2}(e^{- \beta(1-x-y+h_{\beta}(x, y))})\bigg.\\
\bigg.-Li_{2}(e^{-\beta})+Li_{2}(e^{-\beta x })+Li_{2}(e^{- \beta y })+Li_{2}(e^{-\beta(1-x) })+Li_{2}(e^{-\beta(1-y)})\bigg),
\end{multline*}
and the functions $b_{\beta}(x,y)$, $c_{\beta}(x,y)$ and $f_{\beta}(x,y)$ are given by:
\begin{multline*}
b_{\beta}(x,y)=\beta(x+y-2h_{\beta}(x, y))+\ln(1- e^{- \beta(x-h_{\beta}(x, y))})+\ln(1- e^{-\beta(y-h_{\beta}(x, y)) })\\
-\ln(1- e^{- \beta h_{\beta}(x, y)})-\ln(1- e^{- \beta(1-x-y+h_{\beta}(x, y))}),
\end{multline*}
\begin{multline*}
c_{\beta}(x,y)=\frac{\beta^2}{2}(x+y)h_{\beta}(x, y)-\beta^2xy-\frac{1}{2}\ln(1- e^{-\beta h_{\beta}(x, y)})-\frac{\beta x }{2}\ln(1- e^{- \beta(x-h_{\beta}(x, y))})\\
-\frac{ \beta y}{2}\ln(1- e^{- \beta(y-h_{\beta}(x, y))})-\frac{1+ \beta(1-x-y)}{2}\ln(1- e^{-\beta(1-x-y+h_{\beta}(x, y)) })\\
+\frac{1+\beta x }{2}\ln(1- e^{- \beta x })+\frac{1+\beta y}{2}\ln(1- e^{-\beta y })-\left(1+\frac{\beta}{2}\right)\ln(1- e^{-\beta})\\
+\frac{1+\beta(1-x)}{2}\ln(1- e^{-\beta(1-x)})+\frac{1+\beta(1-y)}{2}\ln(1- e^{-\beta(1-y)}),
\end{multline*}
\begin{align*}
f_{\beta}(x,y)=2+\frac{e^{-\beta h_{\beta}(x, y)}}{1-e^{-\beta h_{\beta}(x, y)}}+\frac{e^{-(x-h_{\beta}(x, y)) \beta}}{1-e^{-(x-h_{\beta}(x, y)) \beta}}+\frac{e^{-(y-h_{\beta}(x, y)) \beta}}{1-e^{-(y-h_{\beta}(x, y)) \beta}}+\frac{e^{-(1-x-y+h_{\beta}(x, y)) \beta}}{1-e^{-(1-x-y+h_{\beta}(x, y)) \beta}}.
\end{align*}
Combining the results \eqref{eq:Li2}, \eqref{eq:first-order-derivative}, \eqref{eq:second-order-derivative}, \eqref{eq:d-beta-xy} and \eqref{eq:c-beta-xy}, we have:
\begin{multline*}
\mathcal{M}_{N}^{q} \left( H_{x N, y N}(\omega)=h_{\beta}(x, y) \cdot N +  \alpha \sigma_{N, \beta}(x,y) \right)\\
=\frac{1}{\sigma_{N, \beta}(x,y) \sqrt{2\pi }}\exp\left( - \frac{\alpha^{2}}{2}\right)\left(1+O\left(\frac{\alpha^3}{\sqrt{N}}\right)\right)\left(1+O\left(\frac{\alpha}{\sqrt{N}}\right)\right)\left(1+O\left(\frac{\alpha^4}{N}\right)\right)\\
=\frac{1}{\sigma_{N, \beta}(x,y) \sqrt{2\pi }}\exp\left( - \frac{\alpha^{2}}{2}\right)\left(1+O\left(\frac{\alpha^3}{\sqrt{N}}\right)+O\left(\frac{\alpha}{\sqrt{N}}\right)+O\left(\frac{\alpha^4}{N}\right)\right)\\
\end{multline*}
Letting $k=\lfloor h_{\beta}(x, y) \cdot N +  \alpha \sigma_{N, \beta}(x,y) \rfloor$ such that $|k-h_{\beta}(x, y) N| < \sqrt{N} \cdot A_N$, we get \eqref{eq:LLT} from the above formula by picking the largest error bounds.
\end{proof}

A corollary of Theorem \ref{thm:LLT} is the standard central limit theorem (global central limit theorem) for the height function \eqref{def:Height-Function}.
\begin{corollary}
\label{eq:CLT-global}
Fix a parameter $\beta \in \mathbb{R}_{>0}$, and let $\omega$ be the random permutation distributed according to the Mallows measure $\mathcal{M}_{N}^{q}$, $N \in \mathbb{N}$. Then for any $0 < x, y <1$, $-\infty <a<b<+\infty$, the following convergence holds:
\begin{align}
\label{eq:CLT}
\lim_{N\to\infty}\mathcal{M}_{N}^{q}\left(a < \frac{H_{xN,yN}(\omega)-h_{\beta}(x, y) N}{\sigma_{N, \beta}(x,y)}<b\right)=\int_{a}^{b}\mathcal{N}(\alpha)d\alpha.
\end{align}
\end{corollary}

\begin{proof}
By the triangle inequality,
\begin{align*}
&\left|\mathcal{M}_{N}^{q} \left(a < \frac{H_{xN,yN}(\omega)-h_{\beta}(x, y) N}{\sigma_{N, \beta}(x,y)}<b\right)-\int_{a}^{b}\mathcal{N}(\alpha)d\alpha \right|\\
=&\left| \sum_{k=\lceil a \sigma_{N, \beta}(x,y) + h_{\beta}(x, y) N \rceil}^{\lceil b \sigma_{N, \beta}(x,y) + h_{\beta}(x, y) N-1\rceil} \mathcal{M}_{N}^{q} \left(H_{xN,yN}(\omega)=k\right)-\int_{a}^{b}\mathcal{N}(\alpha)d\alpha \right|\\
\leq & \left| \sum_{k=\lceil a \sigma_{N, \beta}(x,y) + h_{\beta}(x, y) N \rceil}^{\lceil b \sigma_{N, \beta}(x,y) + h_{\beta}(x, y) N-1\rceil} \left(\mathcal{M}_{N}^{q} \left(H_{xN,yN}(\omega)=k\right)-\frac{1}{\sigma_{N, \beta}(x,y)}\mathcal{N}\left(\frac{k-h_{\beta}(x, y) N}{\sigma_{N, \beta}(x,y)}\right)\right) \right|\\
&+\left| \sum_{k=\lceil a \sigma_{N, \beta}(x,y) + h_{\beta}(x, y) N \rceil}^{\lceil b \sigma_{N, \beta}(x,y) + h_{\beta}(x, y) N-1\rceil}\frac{1}{\sigma_{N, \beta}(x,y)}\mathcal{N}\left(\frac{k-h_{\beta}(x, y) N}{\sigma_{N, \beta}(x,y)}\right)- \int_{a}^{b}\mathcal{N}(\alpha)d\alpha\right|.
\end{align*}
The second term in the right-hand side of the above inequality is the difference between the integral of the function $\mathcal{N}$ and the Riemann sum of this integral, therefore this term goes to zero. For the first term, let $c>\max\left\{\frac{a}{\sqrt{\beta}\exp(d_{\beta}(x,y))}, \frac{b}{\sqrt{\beta}\exp(d_{\beta}(x,y))}\right\}$ and we rewrite the result \eqref{eq:LLT} in the local limit theorem:
\begin{align*}
&\max_{|k-h_{\beta}(x, y) N|<c\sqrt{N}}\left|\mathcal{M}_{N}^{q} \left(H_{xN,yN}(\omega)=k\right)-\frac{1}{\sigma_{N, \beta}(x,y)}\mathcal{N}\left(\frac{k-h_{\beta}(x, y) N}{\sigma_{N, \beta}(x,y)}\right)\right|\\
=&\max_{|k-h_{\beta}(x, y) N|<c\sqrt{N}}\left| \frac{\sigma_{N, \beta}(x,y)\mathcal{M}_{N}^{q} \left( H_{x N, y N}(\omega)= k \right)}{\mathcal{N}\left(\frac{k-h_{\beta}(x, y) N}{\sigma_{N, \beta}(x,y)}\right)} - 1\right| \cdot \frac{1}{\sigma_{N, \beta}(x,y)} \cdot \mathcal{N}\left(\frac{k-h_{\beta}(x, y) N}{\sigma_{N, \beta}(x,y)}\right) \\
& \leq O\left(\frac{1}{\sqrt{N}}\right) \cdot \frac{1}{\sigma_{N, \beta}(x,y) \sqrt{2\pi}}=O\left(\frac{1}{N}\right),
\end{align*}
for all $k$'s in the summation from $\lceil a \sigma_{N, \beta}(x,y) + h_{\beta}(x, y) N \rceil$ to $\lceil b \sigma_{N, \beta}(x,y) + h_{\beta}(x, y) N-1\rceil$. The summation has only $O\left(\sqrt{N}\right)$ terms, thus the sum of the error bounds is of the order of $O \left(\frac{1}{\sqrt{N}}\right) = o(1)$.
\end{proof}

\subsection{Local central limit theorem: Several points}

In this section we prove a multi-dimensional local central limit theorem for height functions of the Mallows measure. First, we need a generalization of a result from the previous section.
\begin{lemma}
\label{lem:K-gamma}
Let $\alpha,\gamma \in \R$, let $\beta \in \R_{>0}$, let $x,y \in (0;1)$, let $q=1-\beta/N$, and let $w$ be the random permutation of size $N$ distributed according to the Mallows measure. One has
\begin{multline}
\label{eq:K-gamma}
\lim_{N \to \infty}  \frac{\sqrt{2 \pi N}}{ \sigma_{\beta}(x,y) } \mathcal{M}_{N}^{q} \left( H_{x N, y N+\gamma \sqrt{N}}(\omega)=\delta N + \alpha \sqrt{N}\right)
=\exp \left( -\frac{ \sigma_{\beta}^2(x,y) }{2} \left(\alpha-\mu_{\beta}(x,y;\gamma)\right)^2 \right). 
\end{multline}
with $\mu_{\beta}(x,y;\gamma):=\frac{\sqrt{\beta v_{\beta}(x,y)}}{\sigma_{\beta}(x,y)}\gamma$, where 
\begin{align}
v_{\beta}(x,y):=\frac{e^{-\beta y}(1-e^{-\beta})(1-e^{-\beta x})}{e^{-\beta x}(1-e^{-\beta y})(1-e^{-\beta (1-x)})(1-e^{-\beta (1-y)})}.
\end{align}
\end{lemma}

\begin{proof}
We start from the exact formula \eqref{eq:Height-Function}:
\begin{multline}
\label{eq:formula}
\mathcal{M}_{N}^{q} \left( H_{L,K}(\omega)=s \right)
\\
=\exp\big(\ln q^{(L-s)(K-s)} + \ln (q^{s+1} ; q)_{\infty} + \ln (q^{L-s+1} ; q)_{\infty} + \ln (q^{K-s+1} ; q)_{\infty} + \ln (q^{N+s+1-K-L} ; q)_{\infty} \big.\\
\big. + \ln (q^{N+1} ; q)_{\infty}-\ln (q^{K+1} ; q)_{\infty} - \ln (q^{L+1} ; q)_{\infty} - \ln (q^{N-K+1} ; q)_{\infty} - \ln (q^{N-L+1} ; q)_{\infty} - \ln (q ; q)_{\infty}\big)
\end{multline}
Under the scaling
\begin{align}
\label{eq:scale2}
L=x N,\quad K=y N+\gamma \sqrt{N},\quad q=1-\frac{\beta}{N},\quad s=\delta N+\alpha \sqrt{N}, \quad N \to \infty,
\end{align}
we write down all asymptotic results required in \eqref{eq:formula} with the use of \eqref{eq:lnqq}, \eqref{eq:lnqsqrtNq}, and \eqref{eq:lnqNq}:
\begin{multline*}
\ln q^{(L-s)(K-s)}=(xN-\delta N-\alpha \sqrt{N})(yN+\gamma \sqrt{N}-\delta N-\alpha \sqrt{N}) \cdot \ln \left( 1-\frac{\beta}{N} \right)\\
=-[(x-\delta)(y-\delta)N^2+\gamma(x-\delta)N^{\frac32}-\alpha(x+y-2\delta)N^{\frac32}-\gamma\alpha N+\alpha^2 N] \left( \frac{\beta}{N}+\frac{\beta^2}{2N^2}+O(N^{-3} ) \right)\\
=-\beta(x-\delta)(y-\delta)N+\alpha\beta(x+y-2\delta)N^{\frac12}-\beta\gamma(x-\delta)N^{\frac12}-\frac{\beta^2}{2}(x-\delta)(y-\delta)-\alpha^2\beta+\gamma\alpha\beta+O(N^{-\frac12}).
\end{multline*}
\begin{multline*}
 \ln (q^{s+1} ; q)_{\infty}=-\frac{Li_{2}(e^{-\beta\delta})}{\beta} \cdot N -\alpha\ln(1- e^{-\delta \beta})\cdot \sqrt{N}-\alpha^2 \cdot \frac{\beta e^{-\delta\beta}}{2(1-e^{-\delta\beta})}\\
+\frac12Li_{2}(e^{-\beta\delta})-\frac{1+\delta \beta}{2}\ln(1- e^{-\delta \beta})+O\left(\frac{\alpha^3}{\sqrt{N}}\right)+O\left(\frac{\alpha^4}{N}\right).
\end{multline*}
\begin{multline*}
 \ln (q^{L-s+1} ; q)_{\infty}=-\frac{Li_{2}(e^{-\beta(x-\delta)})}{\beta} \cdot N +\alpha\ln(1- e^{-\beta(x-\delta)})\cdot \sqrt{N}-\alpha^2 \cdot \frac{\beta e^{-\beta(x-\delta)}}{2(1-e^{-\beta(x-\delta)})}\\
+\frac12Li_{2}(e^{-\beta(x-\delta)})-\frac{1+\beta(x-\delta)}{2}\ln(1- e^{-\beta(x-\delta)})+O\left(\frac{\alpha^3}{\sqrt{N}}\right)+O\left(\frac{\alpha^4}{N}\right).
\end{multline*}
\begin{multline*}
 \ln (q^{K-s+1} ; q)_{\infty}=-\frac{Li_{2}(e^{-\beta(y-\delta)})}{\beta} \cdot N -(\gamma-\alpha)\ln(1- e^{-\beta(y-\delta)})\cdot \sqrt{N}-(\gamma-\alpha)^2 \cdot \frac{\beta e^{-\beta(y-\delta)}}{2(1-e^{-\beta(y-\delta)})}\\
+\frac12Li_{2}(e^{-\beta(y-\delta)})-\frac{1+\beta(y-\delta)}{2}\ln(1- e^{-\beta(y-\delta)})+O\left(\frac{(\gamma-\alpha)^3}{\sqrt{N}}\right)+O\left(\frac{(\gamma-\alpha)^4}{N}\right).
\end{multline*}
\begin{multline*}
 \ln (q^{N+s+1-K-L} ; q)_{\infty}=-\frac{Li_{2}(e^{-\beta(1+\delta-x-y)})}{\beta} \cdot N -(\alpha-\gamma)\ln(1- e^{-\beta(1+\delta-x-y)})\cdot \sqrt{N} 
 \\ -(\alpha-\gamma)^2 \cdot \frac{\beta e^{-\beta(1+\delta-x-y)}}{2(1-e^{-\beta(1+\delta-x-y)})}
+\frac12Li_{2}(e^{-\beta(1+\delta-x-y)})
\\ -\frac{1+\beta(1+\delta-x-y)}{2}\ln(1- e^{-\beta(1+\delta-x-y)})+O\left(\frac{(\alpha-\gamma)^3}{\sqrt{N}}\right)+O\left(\frac{(\alpha-\gamma)^4}{N}\right).
\end{multline*}
\begin{align*}
 \ln (q^{N+1} ; q)_{\infty}=-\frac{Li_{2}(e^{-\beta})}{\beta} \cdot N+\frac12Li_{2}(e^{-\beta})-\frac{1+\beta}{2}\ln(1- e^{-\beta})+O(N^{-1}).
\end{align*}
\begin{multline*}
 \ln (q^{K+1} ; q)_{\infty}=-\frac{Li_{2}(e^{-\beta y})}{\beta} \cdot N -\gamma \ln(1- e^{- \beta y})\cdot \sqrt{N}-\gamma^2 \cdot \frac{\beta e^{- \beta y}}{2(1-e^{- \beta y})}\\
+\frac12Li_{2}(e^{-\beta y})-\frac{1+ \beta y}{2}\ln(1- e^{- \beta y})+O\left(\frac{\alpha^3}{\sqrt{N}}\right)+O\left(\frac{\alpha^4}{N}\right).
\end{multline*}
\begin{align*}
 \ln (q^{L+1} ; q)_{\infty}=-\frac{Li_{2}(e^{-\beta x})}{\beta} \cdot N+\frac12Li_{2}(e^{-\beta x})-\frac{1+\beta x }{2}\ln(1- e^{-\beta x })+O(N^{-1}).
\end{align*}
\begin{multline*}
 \ln (q^{N-K+1} ; q)_{\infty}=-\frac{Li_{2}(e^{-\beta(1-y)})}{\beta} \cdot N +\gamma\ln(1- e^{-\beta(1-y)})\cdot \sqrt{N}-\gamma^2 \cdot \frac{\beta e^{- \beta(1-y)}}{2(1-e^{- \beta(1-y)})}\\
+\frac12Li_{2}(e^{-\beta(1-y)})-\frac{1+\beta(1-y)}{2}\ln(1- e^{-\beta(1-y)})+O\left(\frac{\gamma^3}{\sqrt{N}}\right)+O\left(\frac{\gamma^4}{N}\right).
\end{multline*}
\begin{align*}
 \ln (q^{N-L+1} ; q)_{\infty}=-\frac{Li_{2}(e^{-\beta (1-x)})}{\beta} \cdot N+\frac12Li_{2}(e^{-\beta (1-x)})-\frac{1+\beta (1-x) }{2}\ln(1- e^{-\beta (1-x) })+O(N^{-1}).
\end{align*}
\begin{align*}
\ln (q ; q)_{\infty}=-\frac{\pi^2}{6\beta}N+\frac{\pi^2}{12}-\frac12\ln\frac{\beta}{2\pi N}+O\left( N^{-1} \right).
\end{align*}
Using the above asymptotic expansions, we can get the following result:
\begin{multline}
\mathcal{M}_{N}^{q} \left( H_{x N, y N}(\omega)=\delta N + \alpha \sqrt{N}\right)\\
=\frac{\sqrt{\beta}}{\sqrt{2\pi N}}\exp\big(a(x,y,\beta,\delta)N+b(x,y,\beta,\delta,\alpha)\sqrt{N}+c(x,y,\beta,\delta, \alpha)+O\left(\frac{\alpha^3}{\sqrt{N}}\right)\big),
\end{multline}
where we have 
\begin{multline*}
a(x,y,\beta,\delta)=\beta^{-1}\bigg(\frac{\pi^2}{6}-\beta^2(x-\delta)(y-\delta)-Li_{2}(e^{-\beta\delta })-Li_{2}(e^{- \beta(x-\delta)})-Li_{2}(e^{- \beta(y-\delta)})\bigg.\\
\bigg.-Li_{2}(e^{- \beta(1-x-y+\delta)})-Li_{2}(e^{-\beta})+Li_{2}(e^{-\beta x })+Li_{2}(e^{- \beta y })+Li_{2}(e^{-\beta(1-x) })+Li_{2}(e^{-\beta(1-y)})\bigg),
\end{multline*}
\begin{multline*}
b(x,y,\beta,\delta,\alpha,\gamma)=\alpha \beta(x+y-2\delta)-\alpha\ln(1- e^{-\delta \beta})+\alpha\ln(1- e^{-(x-\delta) \beta})+\alpha\ln(1- e^{-(y-\delta) \beta})\\
-\alpha\ln(1- e^{-(1-x-y+\delta) \beta}) -\gamma \beta(x-\delta) -\gamma \ln(1-e^{-(y-\delta) \beta})\\
+\gamma \ln(1- e^{-\beta y}) -\gamma\ln(1-e^{-(1-y) \beta})+\gamma \ln(1- e^{-(1-x-y+\delta)\beta}),
\end{multline*}
\begin{multline*}
c(x,y,\beta,\delta,\alpha,\gamma)=\beta\gamma\alpha\left(1+ \frac{e^{-(y-\delta) \beta}}{1-e^{-(y-\delta) \beta}}+ \frac{e^{-(1-x-y+\delta) \beta}}{1-e^{-(1-x-y+\delta) \beta}}\right)\\
-\frac{\gamma^2\beta}{2}\left(\frac{e^{-(y-\delta) \beta}}{1-e^{-(y-\delta) \beta}}+\frac{e^{-(1-x-y+\delta) \beta}}{1-e^{-(1-x-y+\delta) \beta}}-\frac{e^{-\beta y}}{1-e^{-\beta y}}-\frac{e^{-(1-y) \beta}}{1-e^{-(1-y) \beta}}\right)\\
-\frac{\alpha^2\beta}{2}\left(2+\frac{e^{-\delta\beta}}{1-e^{-\delta\beta}}+\frac{e^{-(x-\delta) \beta}}{1-e^{-(x-\delta) \beta}}+\frac{e^{-(y-\delta) \beta}}{1-e^{-(y-\delta) \beta}}+\frac{e^{-(1-x-y+\delta) \beta}}{1-e^{-(1-x-y+\delta) \beta}}\right)\\
-\frac{1}{2}\bigg(\frac{\pi^2}{6}-\beta^2(x-\delta)(y-\delta)-Li_{2}(e^{-\beta\delta })-Li_{2}(e^{- \beta(x-\delta)})-Li_{2}(e^{- \beta(y-\delta)})-Li_{2}(e^{- \beta(1-x-y+\delta)})\bigg.\\
\bigg.-Li_{2}(e^{-\beta})+Li_{2}(e^{-\beta x })+Li_{2}(e^{- \beta y })+Li_{2}(e^{-\beta(1-x) })+Li_{2}(e^{-\beta(1-y)})\bigg)\\
+\frac{\beta\delta}{2}\big( \beta(x+y-2\delta)+\ln(1- e^{-(x-\delta) \beta})+\ln(1- e^{-(y-\delta) \beta})-\ln(1- e^{-\delta \beta})-\ln(1- e^{-(1-x-y+\delta) \beta})\big)\\
+\bigg(\frac{\beta^2}{2}(x+y)\delta-\beta^2xy-\frac{1}{2}\ln(1- e^{-\delta \beta})-\frac{x \beta}{2}\ln(1- e^{-(x-\delta) \beta})-\frac{y \beta}{2}\ln(1- e^{-(y-\delta) \beta})\bigg.\\
\bigg.-\frac{1+(1-x-y) \beta}{2}\ln(1- e^{-(1-x-y+\delta) \beta})+\frac{1+x \beta}{2}\ln(1- e^{-x \beta})+\frac{1+y\beta}{2}\ln(1- e^{-y \beta})\bigg.\\
\bigg.+\frac{1+(1-x)\beta}{2}\ln(1- e^{-(1-x)\beta})+\frac{1+(1-y)\beta}{2}\ln(1- e^{-(1-y)\beta})-\left(1+\frac{\beta}{2}\right)\ln(1- e^{-\beta})\bigg)\\
+\frac{1}{2}\ln(1- e^{-\beta})-\frac{1}{2}\ln(1- e^{-(x-\delta) \beta})-\frac{1}{2}\ln(1- e^{-(y-\delta) \beta})
\end{multline*}
When $\delta=h_{\beta}(x, y)$ we have
\begin{multline*}
\beta(x-\delta)+ \ln(1-e^{-(y-\delta) \beta})- \ln(1- e^{-\beta y}) +\ln(1-e^{-(1-y) \beta})-\ln(1- e^{-(1-x-y+\delta)\beta})=0.
\end{multline*}
\begin{align*}
1+ \frac{e^{-(y-\delta) \beta}}{1-e^{-(y-\delta) \beta}}+ \frac{e^{-(1-x-y+\delta) \beta}}{1-e^{-(1-x-y+\delta) \beta}}=\frac{1-e^{-(1-x) \beta}}{(1-e^{-(y-h_{\beta}(x, y)) \beta})(1-e^{-(1-x-y+h_{\beta}(x, y)) \beta})}:=u_{\beta}(x,y).
\end{align*}
\begin{multline*}
\frac{e^{-(y-\delta) \beta}}{1-e^{-(y-\delta) \beta}}+\frac{e^{-(1-x-y+\delta) \beta}}{1-e^{-(1-x-y+\delta) \beta}}-\frac{e^{-\beta y}}{1-e^{-\beta y}}-\frac{e^{-(1-y) \beta}}{1-e^{-(1-y) \beta}}\\
=\frac{e^{-\beta y}(1-e^{-\beta})(1-e^{-\beta x})}{e^{-\beta x}(1-e^{-\beta y})(1-e^{-\beta (1-x)})(1-e^{-\beta (1-y)})}:=v_{\beta}(x,y).
\end{multline*}
Combining the results \eqref{eq:Li2}, \eqref{eq:first-order-derivative}, \eqref{eq:second-order-derivative}, \eqref{eq:d-beta-xy} and \eqref{eq:c-beta-xy}, we have:
\begin{multline*}
\mathcal{M}_{N}^{q} \left( H_{x N, y N+\gamma \sqrt{N}}(\omega)=\delta N + \alpha \sqrt{N}\right)\\
=\frac{\sqrt{\beta}}{\sqrt{2\pi N}}\exp\left(\beta\gamma\alpha u_{\beta}(x,y)-\frac{\gamma^2\beta}{2}v_{\beta}(x,y)-\frac{ \sigma_{\beta}^2(x,y) }{2}\alpha^2+O\left(\frac{\alpha^3}{\sqrt{N}}\right)\right),
\end{multline*}
which implies \eqref{eq:K-gamma}.
\end{proof}

Let $r \in \N$, and let $0< y_1 < y_2 < \dots < y_r <1$ be arbitrary real numbers. We define a $r \times r$ symmetric matrix $C_{y_1,y_2, \dots, y_r}$ with entries
$$
C_{y_1,y_2, \dots, y_r}(i,i):= \frac{1}{  \sigma_{\beta} (x,y_i)^2}, \qquad 1 \le i \le r,
$$
$$
C_{y_1,y_2, \dots, y_r} (i,j):= \sqrt{\frac{ e^{-\beta} - e^{-\beta (1-y_i)} }{ \left( 1 -  e^{-\beta (1-y_i)} \right) \sigma_{\beta} (x,y_i)^2 \sigma_{\beta} (x,y_j)^2 }} , \qquad 1\le i < j \le r.
$$

The main goal of this section is to prove the following theorem.
\begin{theorem}
\label{th:CLT-multi}
Let $\alpha_1, \alpha_2, \dots, \alpha_r \in \R$, let $\beta \in \R_{>0}$, let $x \in (0;1)$, let $0< y_1 < y_2 < \dots < y_r <1$, let $q=1-\beta/N$, and let $w$ be the random permutation of size $N$ distributed according to the Mallows measure $\mathcal{M}_{N}^{q} $. One has
\begin{multline*}
\lim_{N \to \infty} \left( \sqrt{2 \pi N} \right)^r \sqrt{\det \left( C_{y_1,y_2, \dots, y_r} \right)}  \mathcal{M}_{N}^{q}  \left( H_{ xN, y_1 N }(\omega) = h_{\beta}(x, y_1) N + \alpha_1 \sqrt{N}, 
\right. \\ \left. H_{xN , y_2 N }(\omega) = h_{\beta}(x, y_2) N + \alpha_2 \sqrt{N}, \dots, H_{ xN , y_r N}(\omega) = h_{\beta}(x, y_r) N + \alpha_r \sqrt{N} \right) 
\\ = \exp \left( \frac{- (\alpha_1, \alpha_2, \dots, \alpha_r) C_{y_1,y_2, \dots, y_r}^{-1} (\alpha_1, \alpha_2, \dots, \alpha_r)^T}{2} \right). 
\end{multline*}

\end{theorem}

\begin{proof}
We have
\begin{multline*}
\mathcal{M}_{N}^{q}  \left( H_{ xN, y_1 N }(\omega) = h_{\beta}(x, y_1) N + \alpha_1 \sqrt{N}, 
\right. \\ \left. H_{xN , y_2 N }(\omega) = h_{\beta}(x, y_2) N + \alpha_2 \sqrt{N}, \dots, H_{ xN , y_r N}(\omega) = h_{\beta}(x, y_r) N + \alpha_r \sqrt{N} \right) 
\\ =
\mathbb{P}_{N}^{q} \left[ H^{(N)}_{y_1 N, y_2 N, \dots, y_r N ;xN} = (h_{\beta}(x, y_1) N + \alpha_1 \sqrt{N} , h_{\beta}(x, y_2) N + \alpha_2 \sqrt{N}
\right. \\ \left. - h_{\beta}(x, y_1) N - \alpha_1 \sqrt{N},
\dots, h_{\beta}(x, y_r) N + \alpha_r \sqrt{N} - h_{\beta}(x, y_{r-1}) N - \alpha_{r-1} \sqrt{N}  ) \right].
\end{multline*}
By Propositon \ref{prop:MalHeightInfManyPoints} the expression above is equal to 
\begin{multline*}
\mathcal{M}_{N}^{q}  \left( H^{(N)}_{x N,y_1 N } = h_{\beta}(x, y_1) N + \alpha_1 \sqrt{N}  \right)
\\ \times \mathcal{M}_{N}^{q}  \left( H^{(N-y_1 N)}_{xN- h_{\beta}(x, y_1) N - \alpha_1 \sqrt{N}, y_2 N -y_1 N } = h_{\beta}(x, y_2) N + \alpha_2 \sqrt{N} - h_{\beta}(x, y_1) N - \alpha_1 \sqrt{N} \right)
\\ \times \cdots 
\\ \times \mathcal{M}_{N}^{q}  \left( H^{(N-y_{r-1} N)}_{xN- h_{\beta}(x, y_{r-1}) N - \alpha_{r-1} \sqrt{N}, y_{r} N -y_{r-1} N } = h_{\beta}(x, y_{r}) N + \alpha_{r} \sqrt{N} - h_{\beta}(x, y_{r-1}) N - \alpha_{r-1} \sqrt{N} \right)
\end{multline*}
Using Lemma \ref{lem:K-gamma}, we obtain that this expression has the following asymptotics as $N \to \infty$:
\begin{multline*}
\frac{\sigma_{\beta}(x,y_1)}{ \sqrt{2 \pi N} } \exp \left( -\frac{ \sigma_{\beta}^2(x,y_1) }{2} \alpha_1^2 \right) \frac{\sigma_{\beta (1-y_1)}( \frac{x - h_{\beta} (x,y_1)}{1-y_1},\frac{y_2-y_1}{1-y_1} ) }{ \sqrt{2 \pi (1-y_1) N} } \\ 
\times \exp \left( -\frac{ \sigma_{\beta (1-y_1)}^2( \frac{x- h_{\beta} (x,y_1)}{1-y_1},\frac{y_2-y_1}{1-y_1} ) }{2} \left( \frac{\alpha_2}{\sqrt{1-y_1}} -\mu_{\beta (1-y_1)} \left( \frac{x- h_{\beta} (x,y_1)}{1-y_1},\frac{y_2 - y_1}{1-y_1};\frac{-\alpha_1}{ \sqrt{1-y_1} } \right) \right)^2 \right) 
\\ \times \cdots
\\ \times \frac{\sigma_{\beta (1-y_{r-1})}( \frac{x - h_{\beta} (x,y_{r-1})}{1-y_{r-1}},\frac{y_r-y_{r-1}}{1-y_{r-1}} ) }{ \sqrt{2 \pi (1-y_{r-1}) N} } \exp \left( -\frac{ \sigma_{\beta (1-y_{r-1} )}^2( \frac{x- h_{\beta} (x,y_{r-1})}{1-y_{r-1}},\frac{y_{r}-y_{r-1}}{1-y_{r-1}} ) }{2} \left( \frac{\alpha_r}{\sqrt{1-y_{r-1}}} 
\right. \right. \\ \left. \left. -\mu_{\beta (1-y_{r-1})} \left( \frac{x- h_{\beta} (x,y_{r-1})}{1-y_{r-1}},\frac{y_{r} - y_{r-1}}{1-y_{r-1}};\frac{-\alpha_{r-1}}{ \sqrt{1-y_{r-1}} } \right) \right)^2 \right). 
\end{multline*}
Using the identities
$$
\det \left( C_{y_1,y_2, \dots, y_r} \right) = \frac{e^{-\beta} - e^{-\beta(1-y_1)}}{1 - e^{-\beta(1-y_1)}} \prod_{i=2}^{r} \left( \frac{e^{-\beta} - e^{-\beta(1-y_{i-1})}}{1 - e^{-\beta(1-y_{i-1})}} - \frac{e^{-\beta} - e^{-\beta(1-y_i)}}{1 - e^{-\beta(1-y_i)}} \right)
\prod_{i=1}^r \frac{1}{\sigma_{\beta}^2 (x,y_i)},
$$
$$
h_{\beta (1-y_{i-1})} \left( \frac{x-h_{\beta}(x, y_{i-1})}{1-y_{i-1}}, \frac{y_i-y_{i-1}}{1-y_{i-1}} \right) = \frac{h_{\beta} (x,y_i) - h_{\beta} (x,y_{i-1})}{1-y_{i-1}},
$$
and also the fact that for arbitrary variables $z_1, z_2, \dots, z_r$ a matrix $\tilde C := \left( z_{\min(i,j)} \right)_{i,j=1}^{r}$ has the tridiagonal inverse with entries
$$
\tilde C^{-1} (i,i)= \frac{1}{z_{i}-z_{i-1}}+\frac{1}{z_{i+1}-z_{i}}, \qquad 1 \le i \le r, \qquad z_0:=0, \ z_{r+1}:=\infty,
$$
$$
\tilde C^{-1} (i,i+1)= \tilde C^{-1} (i+1,i) = \frac{-1}{z_{i+1}-z_{i}}, \qquad 1 \le i \le r-1,
$$
after a direct calculation, one arrives at the statement of the theorem. 

\end{proof}

\begin{remark}
One can obtain a global multi-dimensional central limit theorem from this local one in an analogous way. More importantly, it is plausible that one can analyze the multi-dimensional fluctuations not only in points $(x,y_1), \dots, (x,y_r)$, as we are doing, but in arbitrary points inside the square $(0,1)^2$ via generalizations of Proposition \ref{prop:MalHeightInfManyPoints}. We do not address these questions in the present paper. 

\end{remark}


\begin{thebibliography}{100}

\bibitem[ABC21]{ABC21}
L. Addario-Berry, B. Corsini.
\newblock The height of Mallows trees.
\newblock{{\em The Annals of Probability}}, 2021, 49(5), 2220-2271.

\bibitem[AHHL21]{AHHL21}
O. Angel, A. Holroyd, T. Hutchcroft, A. Levy.
\newblock Mallows permutations as stable matchings.
\newblock{{\em Canadian Journal of Mathematics}}, 2021, 73(6), 1531–1555.

\bibitem[AK24]{AK24}
R. Adamczak, M. Kotowski.
\newblock The global and local limit of the continuous-time Mallows process.
\newblock {\em arXiv preprints}, 2024, arXiv:2404.08554.

\bibitem[And98]{And98}
G. E. Andrews.
\newblock The Theory of Partitions. 
\newblock {\em Cambridge Univ. Press}, 1998.

\bibitem[BB17]{BB17}
R. Basu, N. Bhatnagar.
\newblock Limit theorems for longest monotone subsequences in random Mallows permutations.
\newblock {\em Annales de l'Institut Henri Poincaré, Probabilités et Statistiques}, 2017, 53(4): 1934-1951.

\bibitem[BBHM05]{BBHM05}
I. Benjamini, N. Berger, C. Hoffman, E. Mossel.
\newblock Mixing times of the biased card shuffling and the asymmetric exclusion process.
\newblock {\em Transactions of the American Mathematical Society}, 2005, 357(8): 3013–3029.

\bibitem[BP15]{BP15}
N. Bhatnagar, R. Peled.
\newblock Lengths of monotone subsequences in a Mallows permutation.
\newblock {\em Probability Theory and Related Fields}, 2015, 161(3): 719-780.

\bibitem[BB24]{BB24}
A. Borodin, A.Bufetov.
\newblock ASEP via Mallows coloring
\newblock {\em arXiv preprints}, 2024, arXiv:2408.16585.

\bibitem[BC14]{BC14}
A. Borodin, I. Corwin. 
\newblock Macdonald processes.
\newblock {\em Probability Theory and Related Fields}, 2014, 158(1): 225-400.

\bibitem[BC24]{BC24}
A. Bufetov, K. Chen.
\newblock Mallows product measure.
\newblock {\em arXiv preprints}, 2024, arXiv:2402.09892.

\bibitem[BN22]{BN22}
A. Bufetov, P. Nejjar.
\newblock Cutoff profile of ASEP on a segment.
\newblock {\em Probability Theory and Related Fields}, 2022, 183(1): 229–253.

\bibitem[Buf20]{Buf20}
A. Bufetov.
\newblock Interacting particle systems and random walks on Hecke algebras.
\newblock {\em arXiv preprints}, 2020, arXiv:2003.02730.

\bibitem[Cor22]{Cor22}
B. Corsini.
\newblock Continuous-time Mallows processes.
\newblock {\em arXiv preprints}, 2022, arXiv:2205.04967.

\bibitem[DR00]{DR00}
P. Diaconis, A. Ram.
\newblock Analysis of systematic scan Metropolis algorithms using Iwahori-Hecke algebra techniques.
\newblock {\em Michigan Mathematical Journal}, 2000, 48(1): 157-190.

\bibitem[GK24]{GK24}
V. Gorin, R. Kenyon.
\newblock Six-vertex model with rare corners and random restricted permutations.
\newblock {\em arXiv preprints}, 2024, arXiv:2408.14446.

\bibitem[GO10]{GO10}
A. Gnedin, G. Olshanski. 
\newblock q-Exchangeability via quasi-invariance.
\newblock {\em The Annals of Probability}, 2010, 38(6): 2103-2135.

\bibitem[GO12]{GO12}
A. Gnedin, G. Olshanski.
\newblock The two-sided infinite extension of the Mallows model for random permutations.
\newblock {\em Advances in Applied Mathematics}, 2012, 48(5): 615-639.

\bibitem[GP18]{GP18}
A. Gladkich, R. Peled.
\newblock On the cycle structure of Mallows permutations.
\newblock {\em The Annals of Probability}, 2018, 46(2): 1114–1169.

\bibitem[HHL20]{HHL20}
A. Holroyd, T. Hutchcroft, A. Levy.
\newblock Mallows permutations and finite dependence.
\newblock {\em The Annals of Probability}, 2020, 48(1): 343–379.

\bibitem[HMV23]{HMV23}
J. He, T. Müller, T. Verstraaten.
\newblock Cycles in Mallows random permutations.
\newblock {\em Random Structures \& Algorithms}, 2023, 63(4): 1054-1099.

\bibitem[HR18]{HR18}
G. H. Hardy, S. Ramanujan.
\newblock Asymptotic Formulae in Combinatory Analysis.
\newblock {\em Proc. London Math. Soc.}, 1918, 17(2): 75–115.

\bibitem[HS23]{HS23}
J. He, D. Schmid.
\newblock Limit profile for the ASEP with one open boundary.
\newblock {\em arXiv preprint}, arXiv:2307.14941, 2023.

\bibitem[Kir95]{Kir95}
 A. Kirillov.
\newblock Dilogarithm identities.
\newblock {\em Progress of theoretical physics supplement}, 1995, 118: 61-142.

\bibitem[KKRW20]{KKRW20}
R. Kenyon, D. Kral, C. Radin, P. Winkler.
\newblock Permutations with fixed pattern densities.
\newblock {\em Random Structures \& Algorithms}, 2020, 56(1): 220-250.

\bibitem[Mal57]{Mal57}
C. L. Mallows.
\newblock Non-Null Ranking Models. I.
\newblock {\em Biometrika}, 1957, 44(1/2): 114-130.

\bibitem[MS13]{MS13}
C. Mueller, S. Starr.
\newblock The Length of the Longest Increasing Subsequence of a Random Mallows Permutation.
\newblock {\em Journal of Theoretical Probability}, 2013, 26(2): 514-540.

\bibitem[Moa84]{Moa84}
D. S. Moak.
\newblock The Q-analogue of Stirling’s formula.
\newblock {\em Rocky Mt. J. Math.}, 1984, 14(2), 403–414.

\bibitem[Muk16]{Muk16}
S. Mukherjee.
\newblock Estimation in exponential families on permutations.
\newblock {\em Ann. Statist. }, 2016, 44(2): 853 - 875.

\bibitem[Sta09]{Sta09}
S. Starr.
\newblock Thermodynamic limit for the Mallows model on $S_n$.
\newblock {\em Journal of mathematical physics}, 2009, 50(9):095208.

\bibitem[SW18]{SW18}
S. Starr, M. Walters.
\newblock Phase Uniqueness for the Mallows Measure on Permutations.
\newblock {\em Journal of mathematical physics}, 2018, 59:063301.

\bibitem[Wal15]{Wal15}
M. Walters.
\newblock Concentration of measure techniques and applications.
\newblock {\em Ph.D. thesis}, 2015.


\end{thebibliography}
\end{document}